\documentclass[12pt]{amsart}

\usepackage{amsmath, amssymb, latexsym, amsthm}
\usepackage{fullpage}
\usepackage{verbatim}
\usepackage[flushmargin]{footmisc}

\newtheorem{thm}{Theorem}

\newtheorem{lem}[thm]{Lemma}

\newtheorem{cl}{Claim}

\newenvironment{pf}{\noindent{\bf Proof.}}{\hfill$\Box$ \\ }

\newcommand{\wtilpi}{\widetilde{\pi}}
\newcommand{\lst}{\ell^*}
\newcommand{\htau}{\hat{\tau}}
\newcommand{\case}[1]{\noindent\textbf{#1}}
\newcommand{\hatd}{\hat{d}}

\author[Erbes, Ferrara, Martin, Wenger]{Catherine Erbes$^{1,4}$\and Michael Ferrara$^{1,5}$\and Ryan R. Martin$^{2,6}$ \and Paul Wenger$^3$}
\title{On the approximate shape of degree sequences that are not potentially $H$-graphic}
\date{\today}

\begin{document}
\footnotetext[1]{Department of Mathematical and Statistical Sciences, University of Colorado Denver,{\tt \{catherine.erbes, michael.ferrara\}@ucdenver.edu}}
\footnotetext[2]{Department of Mathematics, Iowa State University, {\tt rymartin@iastate.edu}}
\footnotetext[3]{School of Mathematical Sciences, Rochester Institute of Technology, {\tt pswsma@rit.edu}}
\footnotetext[4]{Research supported in part by UCD GK12 Transforming Experiences Project, National Science Foundation Grant DGE-0742434.}
\footnotetext[5]{Research supported in part by Simons Foundation Collaboration Grant \#206692.}
\footnotetext[6]{Research supported in part by National Science Foundation grant DMS-0901008 and National Security Agency grant H98230-13-1-0226.}

\topmargin 0.4in

\begin{abstract}

A sequence of nonnegative integers $\pi$ is {\it graphic} if it is the degree sequence of some graph $G$.  In this case we say that $G$ is a \textit{realization} of $\pi$, and we write $\pi=\pi(G)$.  A graphic sequence $\pi$ is {\it potentially $H$-graphic} if there is a realization of $\pi$ that contains $H$ as a subgraph.

Given nonincreasing graphic sequences $\pi_1=(d_1,\ldots,d_n)$ and $\pi_2 = (s_1,\ldots,s_n)$, we say that $\pi_1$ {\it majorizes} $\pi_2$ if $d_i \geq s_i$ for all $i$, $1 \leq i \leq n$. In 1970, Erd\H{o}s showed that for any $K_{r+1}$-free graph $H$, there exists an $r$-partite graph $G$ such that $\pi(G)$ majorizes $\pi(H)$. In 2005, Pikhurko and Taraz generalized this notion and  showed that for any graph $F$ with chromatic number $r+1$, the degree sequence of an $F$-free graph is, in an appropriate sense, nearly majorized by the degree sequence of an $r$-partite graph.

  In this paper, we give similar results for degree sequences that are not potentially $H$-graphic.  In particular, there is a graphic sequence $\pi^*(H)$ such that if $\pi$ is a graphic sequence that is not potentially $H$-graphic, then  $\pi$ is close to being majorized by $\pi^*(H)$.  Similar to the role played by complete multipartite graphs in the traditional extremal setting, the sequence $\pi^*(H)$ asymptotically gives the maximum possible sum of a graphic sequence $\pi$ that is not potentially $H$-graphic.

\end{abstract}

\maketitle

\section{Introduction}

All graphs considered in this paper are finite.  A connected component of a graph is \textit{nontrivial} if it has at least one edge.  We let $\Delta(G), \delta(G)$ and $\alpha(G)$ denote the maximum degree, minimum degree and independence number of a graph $G$, respectively and let $\vee$ denote the standard graph join.  Additionally, let $N(v)$ denote the neighborhood of a vertex $v$ in a graph $G$, and for $X\subseteq V(G)$ let $N_X(v)=N(v)\cap X$.  If $H$ is a subgraph of $G$, let $N_H(v)= N_{V(H)}(v)$ and also let $d(v)=|N(v)|$ and $d_X(v)=|N_X(v)|$.

A sequence of nonnegative integers $\pi$ is {\it graphic} if it is the degree sequence of some graph $G$.  Unless otherwise noted, we will assume that all graphic sequences are written in nonincreasing order.  In this case we say that $G$ \textit{realizes $\pi$} or is a \textit{realization} of $\pi$, and we write $\pi=\pi(G)$.  A graphic sequence $\pi$ is {\it potentially $H$-graphic} if there is a realization of $\pi$ that contains $H$ as a subgraph. Let $\pi(H)$ be the degree sequence of the graph $H$. If $\pi(H) = (s_1, \ldots, s_k)$, then we say $\pi = (d_1, \ldots, d_n)$ is {\it degree sufficient for $H$} if $d_i \geq s_i$ for each $i,~ 1 \leq i \leq k$.  When the context is clear, we will write $\pi = (d_1^{m_1}, \dots, d_t^{m_t})$ to indicate that degree $d_i$ has multiplicity $m_i$ in $\pi$.

In this paper, we study the structure of degree sequences that are not potentially $H$-graphic.  As the realizations of a graphic sequence may have a great deal of structural variety, it is perhaps more appropriate to say that we examine the ``shape" (in the Ferrer's diagram sense) of these sequences.  Our inspiration comes from several results on $H$-free graphs from the extremal literature.

Given (not necessarily graphic) sequences $S_1= (x_1, \ldots, x_n)$ and $S_2 = (y_1, \ldots, y_n)$, we say $S_1$ \textit{majorizes} $S_2$ and write $S_1\succeq S_2$ if $x_i \geq y_i$ for all $i$, $1 \leq i \leq n$.  In \cite{Erd}, Erd\H{o}s showed the following.

\begin{thm}\label{theorem:erdos_dom}
If $G$ is a $K_{r+1}$-free graph of order $n$, then there exists an $n$-vertex $r$-partite graph $F$ such that $\pi(F)\succeq \pi(G)$.
\end{thm}

Given positive integers $m$ and $k$, define $D_{k,m}(S_1)$ to be the sequence $$(\underbrace{x_k-m,\dots,x_k-m}_{k~\textrm{times}},x_{k+1}-m,\dots,x_n-m).$$  We say that $S_2$ will \textit{$(k,m)$-majorizes} $S_1$ if $S_2\succeq D_{k,m}(S_1)$. In 2005, Pikhurko and Taraz \cite{PikTar} used this notion to examine the shape of the degree sequences of general $H$-free graphs.

\begin{thm}\label{theorem:pik_tar_dom}
Let $H$ be a graph with chromatic number $\chi(H)=r+1\ge 2.$ For
any $\epsilon > 0$ and $n\ge n_0(\epsilon, H)$, the degree sequence $\pi$ of an $H$-free graph
$G$ of order $n$ is $(\epsilon n, \epsilon n)$-majorized by the degree sequence of some $r$-partite graph of order $n$.
\end{thm}

It was noted in \cite{PikTar} that both the operation of ``leveling off" the first $k$ terms of $\pi(G)$ and the operation of reducing all of the terms in $\pi(G)$ by $m$ are necessary.

That the degree sequences of $r$-partite graphs appear as the bounding class in Theorems \ref{theorem:erdos_dom} and \ref{theorem:pik_tar_dom} is unsurprising given the central role played by the Tur\'{a}n graph $T_{n,r}$, the complete $r$-partite graph of order $n$ with parts as equal as possible, in the extremal literature.

The {\it extremal number}, denoted $ex(H,n)$, is the maximum number of edges in a graph of order $n$ that does not contain $H$ as a subgraph.  While the exact value of the extremal function is known for very few graphs (cf. \cite{BK,CGPW,EFGG, Tur41}), in 1966 Erd\H{o}s and Simonovits \cite{ESS} extended previous work of Erd\H{o}s and Stone \cite{ES} and determined $ex(H, n)$ asymptotically for arbitrary $H$.  More precisely, this seminal theorem gives exact asymptotics for $ex(H,n)$ when $H$ is a nonbipartite graph.

\begin{thm}[The Erd\H{o}s-Stone-Simonovits Theorem]
If $H$ is a graph with chromatic number $\chi(H)=r+1\ge 2$, then $$ex(H,n)=\max\{|E(G)| \colon |G|=n, H\not\subseteq G\}=|E(T_{n,r})| + o(n^2).$$
\end{thm}

It is our goal to examine the structure of degree sequences that are not potentially $H$-graphic in a manner similar to Theorems \ref{theorem:erdos_dom} and \ref{theorem:pik_tar_dom}.
In order to do so, we will next discuss a recent result on potentially $H$-graphic sequences that will allow us to identify the sequences that form our bounding class. For a graphic sequence $\pi$, we let $\sigma(\pi)$ denote the sum of the terms of $\pi$.

\subsection{The Potential Function $\sigma(H,n)$}

In 1991, Erd\H{o}s, Jacobson and Lehel \cite{EJL} proposed the following problem:\\

\begin{center}
 Determine $\sigma(H,n)$, the minimum even integer such that every $n$-term graphic sequence $\pi$ with $\sigma(\pi)\ge \sigma(H,n)$ is potentially $H$-graphic.\\
\end{center}

  We refer to $\sigma(H,n)$ as the {\it potential number} or {\it potential function} of $H$.  As $\sigma(\pi)$ is twice the number of edges in any realization of $\pi$, the Erd\H{o}s-Jacobson-Lehel problem can be viewed as a degree sequence relaxation of the Tur\'{a}n problem.  While the exact value of $\sigma(H,n)$ has been determined for a number of specific graph classes (c.f. \cite{ChenLiYin,EJL,F,FS,LiXia,LiYin2}), little was known about the behavior of the potential function for general graphs until recently when Ferrara, LeSaulnier, Moffatt and Wenger ~ \cite{FLMW} determined $\sigma(H,n)$ asymptotically for all $H$.  We describe their result next.

Let $H$ be a graph on $k$ vertices with at least one nontrivial connected component.  For each $i\in\{\alpha(H)+1,\ldots, k\}$, define
\[
\nabla_i(H)=\min\left\{\Delta(F) : F\leq H, |V(F)|=i\right\},
\]
where $F \leq H$ denotes that $F$ is an induced subgraph of $H$.
Let $n$ be sufficiently large, and consider the sequence
\[
\widetilde{\pi}_i(H,n) = ((n-1)^{k-i}, (k-i + \nabla_i(H)-1)^{n-k+i}).
\]

\noindent This sequence is graphic provided that $n-k+i$ and $\nabla_i(H) -1$ are not both odd. If they are both odd, then reduce the last term of the sequence by 1. As given in \cite{FLMW}, the resulting sequence is graphic, but not potentially $H$-graphic.  Consequently, $\sigma(H,n)\ge \max_i \left(\sigma(\widetilde{\pi}_i(H,n)\right)$. This maximum is attained by $\widetilde{\pi}_{i^*}(H,n)$, where we define $i^*=i^*(H)$ to be the smallest index $i$ in $\{\alpha(H)+1,\dots,k\}$ that minimizes the quantity $2i-\nabla_i(H)$ and therefore maximizes $\sigma(\widetilde{\pi}_i(H,n))$.  The main result of \cite{FLMW} states that $\widetilde{\pi}_{i^*}(H,n)$ determines $\sigma(H,n)$ asymptotically for all $H$, which can be viewed as an Erd\H{o}s-Stone-Simonovits-type theorem for the Erd\H{o}s-Jacobson-Lehel problem.

\begin{thm}[Ferrara, LeSaulnier, Moffatt and Wenger~\cite{FLMW}]

If $H$ is a graph and $n$ is a positive integer, then $$\sigma(H,n) = \sigma(\widetilde{\pi}_{i^*}(H,n)) + o(n).$$
\end{thm}

\subsection{Main Results.} Through the remainder of this paper, unless otherwise noted we will assume that all sequences have minimum term at least 1. Given two $n$-term graphic sequences $\pi_1=(d_1, \ldots, d_n)$ and $\pi_2$, and nonnegative integers $a_1, a_2$, and $b$ with $a_1 \leq a_2$, we say that $\pi_1$ is $([a_1,a_2], b)$-\textit{close} to $\pi_2$ if there is a (not necessarily graphic) sequence $\pi_1'$ with $\pi_2\succeq \pi_1'$ such that $\pi_1'$ can be obtained from $\pi_1$ via the following two steps:

\begin{enumerate}

\item Create the sequence
\[
S_1=(d_1, \ldots, d_{a_1-1}, \underbrace{d_{a_2}, \ldots, d_{a_2}}_{a_2-a_1+1~\textrm{times}}, d_{a_2+1}, \ldots, d_n).
\]

\item Create $\pi_1'$ from $S_1$ by subtracting a \textit{total} of at most $b$ from the terms of $S_1$.

\end{enumerate}

We will refer to step (1) as {\it leveling off} terms $a_1$ to $a_2$ of $\pi_1$ and the procedure in step (2) as {\it editing} the sequence $S_1$.

In contrast to the idea of $(k,m)$-majorization, $([a_1,a_2],b)$-closeness leaves the first $a_1 -1$ terms unchanged, and after the leveling off step allows for variable editing, provided that the total amount of editing in step (2) is at most $b$.

We show that if a sequence is not potentially $H$-graphic, then it is close (in the above sense) to being majorized by one of the sequences $\widetilde{\pi_i}(H,n)$. Our main results are as follows.

\begin{thm}\label{thm:NotDegSuff}
Let $H$ be a graph with degree sequence $\pi(H) = (h_1, \ldots, h_k)$, and let $\pi = (d_1, \ldots, d_n)$ be a graphic sequence that is not degree sufficient for $H$. Further, let $j$ be the largest integer for which $d_{k-j+1} < h_{k-j+1}$. If $j \geq \alpha(H)+1$, then $\pi$ is majorized by $\wtilpi_{j}(H,n)$. If $j < \alpha(H)+1$, then $\pi$ is $([k-\alpha(H), k-j+1], 0)$-close to $\wtilpi_{\alpha(H)+1}(H,n)$.  This result is best possible.
\end{thm}

The case where $\pi$ is degree sufficient for $H$ seems to be much more technical, and requires both the leveling off and editing operations outlined above.  Recall that $i^*=i^*(H)$ is the smallest $i$ in $\{\alpha(H)+1,\dots,k\}$ that minimizes $2i-\nabla_i(H)$.

\begin{thm}\label{thm:DegSuff}
Let $H$ be a graph of order $k$ with at least one nontrivial component and let $\pi$ be an $n$-term graphic sequence that is degree sufficient for $H$. If $\pi$ is not potentially $H$-graphic, then $\pi$ is $([k-i^*+1, k], (6\alpha~ +~3) k^2+ \alpha^3 k)$-close to $\wtilpi_{i^*}(H,n)$.
\end{thm}

In Section~\ref{sharpness}, we provide an example of a graph $H$ and sequence that is not potentially $H$-graphic and requires $O(\alpha k^2)$ editing, proving that Theorem~\ref{thm:DegSuff} is in some sense best possible up to the coefficient of $\alpha k^2$.
However, this sharpness remains an open question in those cases when $\alpha>>k^{1/2}$.

\section{Preliminaries}

The following results will be used repeatedly throughout our proofs of Theorems \ref{thm:NotDegSuff} and \ref{thm:DegSuff}. The first is the well-known characterization of graphic sequences due to Erd\H{o}s and Gallai.

\begin{thm}[Erd\H{o}s and Gallai~\cite{ErdGal}] \label{thm:ErdGal}
A degree sequence $\pi = (d_1, \ldots, d_n)$ such that $d_1 \geq \cdots \geq d_n$  is graphic if and only if $\sum_{i=1}^n d_i$ is even and, for all $p\in\{1,\ldots,n-1\}$,
\begin{equation}
   \sum_{i=1}^p d_i\leq p(p-1)+\sum_{i=p+1}^n\min\{d_i,p\} . \label{eq:ErdGal}
\end{equation}
\end{thm}

The following Lemma is central to the proof of Theorem \ref{thm:DegSuff}, and is likely also of independent interest.  As the proof of this result is quite technical, we postpone it until Section \ref{section:biclique_proof}.

\begin{lem}\label{biclique}
Let $r$ and $k$ be positive integers with $r< k$, and let $\pi = (d_1, \ldots, d_n)$ be a graphic sequence. Suppose that $d_{k-r} - d_{k} \geq r(k+2)$. If there are at least $r(k+r+1)$ terms among $d_{k+1}, \ldots, d_n$ with values in $\{k-r, \ldots, k-1\}$,  then $\pi$ has a realization containing the graph $K_{k-r, r}$ with vertices of degree $d_1, \ldots, d_{k-r}$ forming the partite set of order $k-r$.
\end{lem}

The following gives a bound on the length of a sequence with fixed maximum term that is not potentially $H$-graphic.

\begin{lem}\label{lem:Hgraph}
Let $H$ be a graph with $\pi(H)=(h_1,\dots,h_k)$ and let $\pi = (d_1, \ldots, d_n)$ be a graphic sequence with $d_1\le M$ such that there are terms $d_{i_1},\dots,d_{i_k}$ of $\pi$ satisfying $d_{i_j}\ge h_j$ for $1\le j\le k$.  If $\pi$ has at least $2M^2 + k$ positive terms, then there is a realization $G$ of $\pi$ with a copy of $H$ that lies on vertices of degree $d_{i_1},\dots,d_{i_k}$.
\end{lem}
\begin{pf}
We may assume that $d_{i_j} = d_j$ for all $j$ and also that $n\ge 2M^2 + k$ and $d_n\ge 1$.  First note that if $M=1$, then $H$ must be a set of disjoint edges and isolated vertices, and $\pi$ is potentially $H$-graphic as long as $n \geq k$.  We therefore assume $M \geq 2$.

Let $V(H) = \{u_1, \ldots, u_k\}$, with the vertices in nonincreasing order by degree. In a realization $G$ of $\pi$, let $S = \{v_1, \ldots, v_k\}$ be the vertices with the $k$ highest degrees (in order) and let $H_S$ be the graph with vertex set $S$ and $v_i v_j \in E(H_S)$ if and only if $u_i u_j \in E(H)$. If all of the edges of $H_S$ are in $G$, then $H_S$ is a subgraph of $G$ that is isomorphic to $H$.

Assume now that $G$ is a realization of $\pi$ that maximizes $|E(H_S) \cap E(G)|$, but this quantity is less than  $|E(H_S)|$. Thus, there exist $v_i, v_j \in V(G)$ such that $v_i v_j \notin E(G)$ but $v_i v_j \in E(H_S)$. Since $\pi$ is degree sufficient for $H$, it follows that $v_i$ and $v_j$ must each have a neighbor, say $a_i$ and $a_j$, respectively, such that $v_i a_i, v_j a_j \notin E(H_S)$ but $v_i a_i, v_j a_j \in E(G)$. Note that possibly $a_i = a_j$.

Since the maximum degree in $G$ is $M$, there are at most $M^2+1$ vertices at distance at most 2 from $a_i$, and at most $M^2+1$ vertices at distance at most 2 from $a_j$. Since $a_i$ and $a_j$ have distinct neighbors in $S$, there are at most $k-2$ vertices in $S$ that are distance at least 3 from both $a_i$ and $a_j$. Therefore, there is a vertex $w$ in $V(G) \backslash S$ that is distance at least 3 from both $a_i$ and $a_j$. Let $x$ be a neighbor of $w$; consequently $x$ is not adjacent to $a_i$ or $a_j$, and $xw \notin E(H_S)$.
Exchanging the edges $v_i a_i, v_j a_j$, and $wx$ for the non-edges $v_i v_j, a_i w$, and $a_j x$ yields a realization $G'$ of $\pi$ such that $|E(H_S) \cap E(G')| > |E(H_S) \cap E(G)|$, contradicting the maximality of $G$.
\end{pf}

Finally, the following theorem of Li and Yin gives useful sufficient conditions for a degree sequence to be potentially $K_k$-graphic. We will use this repeatedly in our proofs of Theorems \ref{thm:NotDegSuff} and \ref{thm:DegSuff}.

\begin{thm}[Li and Yin, \cite{LiYin}]\label{thm:YinLi}
Let $\pi = (d_1, \ldots, d_n)$ be a nonincreasing graphic sequence and let $k$ be a positive integer.
\begin{itemize}
\item[(a)] If $d_k \geq k-1$ and $d_i \geq 2(k-1) -i$ for $1 \leq i \leq k-2$, then $\pi$ is potentially $K_k$-graphic.
\item[(b)] If $d_k \geq k-1$ and $d_{2k} \geq k-2$, then $\pi$ is potentially $K_k$-graphic.
\end{itemize}
\end{thm}

\section{Sharpness}\label{sharpness}

Prior to proving Theorems \ref{thm:NotDegSuff} and \ref{thm:DegSuff}, we will discuss the sharpness of these results.

The number of terms leveled off in Theorem \ref{thm:NotDegSuff} is best possible in light of the following example. Let $H = K_{k-r} \vee \overline{K}_r$, where $r$ is at least 2, and for $1\le j < \alpha(H)+1$ let
$$\pi_j = \left(\left(\frac{n}{k+1}\right)^{k-j}, (k-r-1)^{n-k+j}\right)$$
where $n$ is sufficiently large. If the sum of $\pi_j$ is even, then $\pi_j$ is graphic; if the sum is odd, then reducing the last term by 1 yields a graphic sequence.  Clearly $\pi_j$ is not degree sufficient for $H$.

Note that the $(k-j+1)^{\rm st}$ term of $\pi_j$ is the first place that degree sufficiency for $H$ fails. Since $k-r-1 \leq k-2$, the last $n-k+j$ terms of $\pi_j$ are termwise dominated by the last $n-k+j$ terms of $\wtilpi_{\alpha(H)+1}(H,n)$.  However, $\pi_j$ has $k-j$ large terms, of which only the first $k- \alpha(H)-1$ are dominated by  $\wtilpi_{\alpha(H)+1}(H,n)$.  Therefore, we need to reduce terms $d_{k-\alpha(H)}, \ldots, d_{k-j}$ of $\pi_j$ to $d_{k-j+1} = k-r-1$, and each of these reductions is on the order of $n$. This yields a sequence that is majorized by $\wtilpi_{\alpha(H)+1}(H,n)$, but reducing any smaller number of terms would not suffice.



To evaluate the sharpness of Theorem \ref{thm:DegSuff}, we consider $H=K_k$.
Let
\[
\pi_k=\left(n-1,(2k-5)^{2k-3},1^{n-2k+2}\right)
\]
The Erd\H os-Gallai criteria show that $\pi_k$ is graphic;  clearly $\pi_k$ is degree sufficient for $K_{k}$.
Observe that $\pi_k$ is potentially $K_k$-graphic if and only if the sequence $\pi_k' = \left((2k-6)^{2k-3}\right)$, obtained by performing the Havel-Hakimi algorithm, is potentially $K_{k-1}$-graphic.
However, the complement of any realization of $\pi_k'$ is a $2$-regular graph, so the maximum size of a clique in any realization of $\pi_k'$ is at most $k-2$.
Hence $\pi_k$ is not potentially $K_{k}$-graphic.

Since $2i - \nabla_i(K_{k}) = i+1$ for each $i \in \{2, \ldots, k\}$, we have $i^*=2$.
Note that $\wtilpi_2(K_{k},n) = ((n-1)^{k-2}, (k-2)^{n-k+3})$.
Leveling off terms $k-i^*+1=k-1$ through $k$ of $\pi_k'$ does not change the sequence.
However, each entry from $k-1$ through $2k-2$ is larger than $k-2$, so we need to reduce each of these entries by $k-3$, for a total of $k(k-3)$ editing.
Thus we perform a total of $O(\alpha k^2)$ editing.

\section{Proofs of Theorems \ref{thm:NotDegSuff} and \ref{thm:DegSuff}}

\noindent {\bf Proof of Theorem \ref{thm:NotDegSuff}.}
First note that for each $ i \ge \alpha(H)+1$, $h_{k-i+1} \le k - i + \nabla_i(H)$
Otherwise, every $i$-vertex induced subgraph of $H$ has maximum degree greater than $\nabla_i(H)$, contradicting the definition of $\nabla_i(H)$.
%

Similarly, if $i <\alpha(H)+1$, then $h_{k-i+1} \leq k-\alpha(H)$.
Otherwise, at most $i-1<\alpha(H)$ vertices have degree at most $k-\alpha(H)$, while there are at least $\alpha(H)$ vertices in $H$ with degree at most $k-\alpha(H)$.

Recall that $j$ is the largest integer for which $d_{k-j+1} < h_{k-j+1}$.
First suppose that $j\ge \alpha(H)+ 1$.
In this case, we show that $\pi$ is majorized by $\wtilpi_j(H,n)$.
Clearly, the first $k-j$ terms of $\pi$ are majorized by the first $k-j$ terms of $\wtilpi_j(H,n)$.
As $d_{k-j+1}<h_{k-j+1} \leq k-j + \nabla_j(H)$, the remaining terms of $\pi$ are at most $k-j+\nabla_j(H) -1$. Thus, $\pi$ is majorized by $\wtilpi_j(H,n)$.

Now suppose that $j <\alpha(H)+1$. Here we show that $\pi$ is $([k-\alpha(H), k-j+1], 0)$-close to $\wtilpi_{\alpha(H)+1}(H,n)$. We know that $d_{k-j+1} < h_{k-j+1} \leq k - \alpha(H)$. Since $k-\alpha(H) + \nabla_{\alpha(H)+1}-2 \geq k-\alpha(H)-1$, reducing terms $d_{k-\alpha(H)}$ through $d_{k-j}$ of $\pi$ to $d_{k-j+1}$ results in a sequence that is majorized by $\wtilpi_{\alpha(H)+1}(H,n)$. \hfill $\Box$~\\

For the proof of Theorem~\ref{thm:DegSuff}, we prove a more technical result that follows below.
First we define some terminology that is used in the proof.

Given a graphic sequence $\pi$ that is degree sufficient for $K_r \vee \overline{K}_{k-r}$, we will create a sequence $\pi^w$, called the \textit{want sequence of $\pi$ for $K_r \vee \overline{K}_{k-r}$}. Begin by finding a realization $G$ of $\pi$ on the vertices $\{v_1, \ldots, v_n \}$ with $d(v_i) = d_i$ that maximizes the sum of (a) the number of edges amongst $v_1, \ldots, v_r$ and (b) the number of edges joining $\{v_1, \ldots, v_r\}$ and $\{v_{r+1}, \ldots, v_{k} \}$. Let $G_r = G[v_{r+1}, \ldots, v_n]$, and let $\pi_0^w = (w_{r+1}, \ldots, w_n)$ be the degree sequence of $G_r$, indexed so that $w_i = d_{G_r}(v_i)$.

For each $v_i$ with $i \leq r$, we want $v_i$ to be adjacent to each of the vertices in the set $S_i = \{v_1, \ldots, v_{k}\} \backslash \{v_i\}$.
Since $\pi$ is degree sufficient for $K_r \vee \overline{K}_{k-r}$, we see that for each nonneighbor of $v_i$ in $S_i$, there is a neighbor of $v_i$ in $\{v_{k+1}, \ldots, v_n\}$, and each of these neighbors is distinct.
Let $W_i$ be a subset of $N_{G_r}(v_i) \cap \{v_{k+1}, \ldots, v_n\}$ that has size $k-1 - d_{S_i}(v_i)$. Let $W$ be the multiset $\cup_{i=1}^k W_i$.

To create the want sequence from $\pi_0^w$, we make the following modifications. Each time the vertex $v_y$ appears in $W$, add 1 to entry $w_y$ of $\pi_0^w$. For each $j$ with $r+1 \leq j \leq k$, subtract $r-d_{\{v_1, \ldots, v_r\} }(v_j)$ from $w_j$. The sequence that results from these modifications is the want sequence, $\pi^w$. Note that the largest value that can be subtracted from any entry is $r$, and since $\pi$ is degree-sufficient for $K_r \vee \overline{K}_{k-r}$ and the only entries that might be reduced are those with index at most $k$, no entry of $\pi^w$ is negative and at most $r$ terms of $\pi^w$ are $0$. The largest value that will be added to any entry of $\pi_0^w$ is at most $r$, and only terms with index at least $k+1$ are increased, so the largest entry of $\pi^w$ is at most the maximum of $w_{r+1}$ and $w_{k+1} + r$.

If $\pi^w$ is graphic, we can find a realization of $\pi$ that contains $K_r \vee \overline{K}_{k-r}$.
To do this, take the union of the complete split graph $K_r \vee \overline{K}_{k-r}$ on the vertices $\{v_1, \ldots, v_k\}$ (with the clique on the vertex set $\{v_1\ldots,v_r\}$) and a realization of $\pi^w$ on the vertices $\{v_{r+1}, \ldots, v_n\}$.
Then join each vertex belonging to the clique of the complete split graph (that is, $v_i$ such that $i \leq r$) to the vertices in $N(v_i) \cap \{v_{k+1}, \ldots, v_n\} \backslash W_i$. This graph has degree sequence $\pi$, so we have a realization of $\pi$ that contains the desired complete split graph.

We will prove the following, more specific result than that stated in Theorem \ref{thm:DegSuff}.

\begin{thm}\label{thm:DegSuffNab}
Let $H$ be a fixed graph of order $k$ with at least one non-trivial component. If $\pi$ is a graphic sequence of length $n$ that is degree sufficient for $H$ but not potentially $H$-graphic, then $\pi$ is
\[
([k-i^*+1,k], k^2+k i^* + 2 + (6k^2+(i^*-\nabla_{i^*}(H))^2k+\nabla_{i^*}(H))(i^*-\nabla_{i^*}(H)-2))\text{-close}
\]
to $\wtilpi_{i^*}(H,n)$.
\end{thm}

Since
\[
2i^*-\nabla_{i^*}(H) \le 2(\alpha(H)+1)-\nabla_{\alpha(H)+1}(H)\leq 2 \alpha(H)+1,
\]
and $i^*\ge \alpha(H)+1$, we see that $i^* - \nabla_{i^*}(H) \leq \alpha(H)$.
Thus, Theorem \ref{thm:DegSuff} follows directly from Theorem~\ref{thm:DegSuffNab}.\\

\begin{pf}
Let $\pi=(d_1, \ldots, d_n)$ and label the vertices of $H$ with $\{v_1, \ldots, v_k\}$ such that $d(v_i) \geq d(v_j)$ when $i < j$. To simplify notation, we will write $\alpha$ for $\alpha(H)$ and let $\ell^* = i^*-\nabla_{i^*}(H)$. Let $f(H) = 6k^2 + (\lst)^2 k +\nabla_{i^*}(H)$. With this notation, we prove that $\pi$ is $([k-i^*+1,k], k^2+k i^* + 2 + f(H)(\lst-2))\text{-close}$
to $\wtilpi_{i^*}(H,n)$.

If either $d_k\ge 2k-3$ or $d_{2k}\ge k-2$, then by Theorem \ref{thm:YinLi}, $\pi$ is potentially $K_k$-graphic.  As this would imply that $\pi$ is potentially $H$-graphic, we assume henceforth that $d_k\le 2k-4$ and, if $n\ge 2k$, that $d_{2k}\le k-3$.

We will break the proof into several cases. In Cases 1-3, we will show that, after reducing the value of terms $d_{k-i^*+1}, \ldots, d_{k-1}$ to $d_k$, we only require an editing of at most $k^2+ki^*+2+f(H)(\ell^*-2)$. In Case 4, no editing is required, and we show that $\pi$ is potentially $H$-graphic.\\

\case{Case 1} $n < 2k$. ~\\

In this case, after leveling off terms $d_{k-i^*+1}$ through $d_{k-1}$, we need to reduce at most $k+i^*$ terms of the sequence. Each of those terms is reduced by at most $k + \lst -3$, so the total amount of editing is at most $k^2 + k i^* + (\lst -3)(k+i^*)$.\\

\case{Case 2} $2k \leq n < f(H)$. ~\\

Reducing the terms $d_{k-i^*+1}, \ldots, d_{k-1}$ of $\pi$ to $d_k$ leaves at most $k-i^*$ terms that may be greater than $2k-4$.
Now in the resulting sequence, each of the terms from position $k-i^*+1$ to position $2k-1$ needs to be reduced by at most $k+\lst -3$ and each term from position $2k$ to the end of the sequence must be reduced by at most $\lst-2$. Thus the amount of editing required on this sequence is at most
\begin{equation*}\label{eq:editbd}
(k+i^*-1)(k+\lst -3)+(n-2k+1)(\lst-2) \leq (\lst -2)(f(H)-k+i^*) + (k-1)(k+i^*-1).
\end{equation*}
~\\

\case{Case 3} $n \geq f(H)$ and $d_{f(H)} \leq k- \lst - 1$.~\\

Here, we can stop editing after the term $d_{f(H)}$ since all subsequent terms are at most $k-\lst-1$.
We perform the same leveling-off step as in Case 2, so the amount of editing required is at most
\begin{equation*}
(k+i^*-1)(k+\lst-3)+(f(H)-2k+1)(\lst-2)=(\lst-2)(f(H)-k+i^*)+(k-1)(k+i^*-1).
\end{equation*}~\\

\case{Case 4} $n \geq f(H)$ and $d_{f(H)} > k-\lst-1$.\\

Now we show that $\pi$ is potentially $H$-graphic.
If $\lst = 1$, then $d_{f(H)} \geq k-1$, and since $f(H) \geq 2k$, this means $d_{2k} \geq k-1$. Thus by part (b) of Theorem \ref{thm:YinLi}, $\pi$ is potentially $K_k$-graphic. We assume henceforth that $\lst \geq 2$.

Let $t$ be such that $d_t \geq k-1$ but $d_{t+1} < k-1$. We then have two cases.\\

\case{Case 4a} $t < k - \lst$.\\

In this case, we wish to show that $\pi$ has a realization that contains the complete split graph $K_t \vee \overline{K}_{k-t}$.
First note that $\pi$ is degree sufficient for such a graph because $d_t \geq k-1$ and $d_k \geq k - \lst > t$. Let $\pi_1^w$ be the want sequence of $\pi$ for $K_t \vee \overline{K}_{k-t}$. Since every entry of $\pi_0^w$ is at most $k-2$ (because $d_{t+1} < k-1$), and $t$ is also at most $k-1$, the largest entry of $\pi_1^w$ is less than $2k$.

Zverovich and Zverovich \cite{ZZ} showed that a sequence with maximum term $r$ and minimum term $s$ is graphic as long as the length of the sequence is at least $\frac{(r+s+1)^2}{4s}$. Since $\pi_1^w$ has length $n-t$ and up to $t$ terms may be 0, $\pi_1^w$ is graphic if $n - 2t \geq (k+1)^2$. Since $t < k-\lst$, this is true if $n \geq k^2+4k+1 - 2\lst$. The observation that $n \geq f(H) > 6k^2$ shows that this is true, and $\pi_1^w$ is graphic.

Now observe that if $\pi_1^w$ is graphic, then $\pi$ is potentially $K_t \vee \overline{K}_{k-t}$-graphic.
If $t \geq k-\alpha(H)$, then this complete split graph contains a copy of $H$, and we are done. So we assume that $t < k - \alpha(H)$. Let $F_i$ denote an $i$-vertex induced subgraph of $H$ that achieves $\Delta(F_i) = \nabla_i(H)$.  If $\pi$ has a realization containing $K_t\vee \overline{K}_{k-t}$ with a copy of $F_{k-t}$ on the vertices in the independent set, then $\pi$ is potentially $H$-graphic.

By Lemma 2.1 of \cite{Yin}, we know that a realization $G$ of $\pi$ can be found that contains this split graph with the following property: the $t$ vertices of degree $k-1$ are on the $t$ highest-degree vertices of $G$, and the $k-t$ vertices of degree $t$ are on the next $k-t$ highest degree vertices of $G$.
Delete the $t$ vertices of highest degree in $G$, and let $\pi' = (d_1', \ldots, d_{n-t}')$ be the degree sequence of the resulting subgraph of $G$.
It follows that $\pi'$ satisfies $d_1' \leq k-2$ and $d_{f(H)-t}' \geq k-\lst-t \geq 1$. Since $\Delta(F_{k-t})=\nabla_{k-t}(H)$ and $d_{k-t}' \geq d_{f(H)-t}' \geq k-\lst \geq \nabla_{k-t}(H)$, it follows that $\pi'$ is degree sufficient for $F_{k-t}$. Applying Lemma \ref{lem:Hgraph} with $H = F_{k-t}$ and $M = k-2$, we see that $\pi'$ is potentially $F_{k-t}$-graphic as long as at least $2(k-2)^2 + (k-t)$ terms of $\pi'$ are positive. Since $f(H)-t \geq 6k^2$, there is a realization of $\pi'$ that contains a copy of $F_{k-t}$ on the highest degree vertices; overlapping this with the vertices $v_{t+1}, \ldots, v_k$ of $G$, we get a realization of $\pi$ containing $K_{t} \vee F_{k-t}$, which implies that $\pi$ is potentially $H$-graphic. \\

\case{Case 4b} $t \geq k - \lst$.\\

First, suppose $d_{k-\ell^*}-d_k \geq \lst(k+2)$.
This implies that $d_{k-\lst} \geq \lst (k + 2)+ d_k$.
Since $\lst \geq 2$, it follows that $d_{k-\lst} \geq 3k$.
We claim that this implies that $\pi$ is potentially $H$-graphic.

Since $f(H) > \lst (k + \lst +1)$ and $d_{f(H)} \geq k-\lst $, Lemma \ref{biclique} yields a realization $G$ of $\pi$ containing the complete bipartite graph $K_{k-\lst, \lst}$, where the vertices $\{v_1, \ldots, v_{k-\lst}\}$ form the partite set of order $k-\lst$.
Let $S$ be the set of vertices in the copy of $K_{k-\lst, \lst}$, let $S' = \{v_1, \ldots, v_{k-\lst}\}$, and let $R = V(G)\setminus S$.
If the vertices of $S'$ induce a complete graph, then $G$ contains $K_{k-\lst} \vee \overline{K}_{\lst}$; consequently $G$ contains a copy of $H$ since $k-\lst \geq k-\alpha$.
Suppose there are vertices $v_i$ and $v_j$ in $S'$ such that $v_iv_j\notin E(G)$.
Since $d_S(v_i) \leq k-2$ and $d(v_i) \geq 3k-2$, we know that $v_i$ has at least $2k$ neighbors in $R$.
Similarly, $v_j$ has at least $2k$ neighbors in $R$.
If each neighbor of $v_i$ in $R$ is adjacent to each neighbor of $v_j$ in $R$, then each of these vertices in $R$ has degree at least $2k-1$.
Since $v_i$ has at least $2k$ neighbors in $R$, it follows that $G$ has at least $2k$ vertices with degree at least $2k-1$, contradicting the assumption that $d_k\le 2k-4$.
Thus there are vertices $x$ and $y$ in $R$ such that $v_i x, v_j y\in E(G)$, and $x y\notin E(G)$.
Hence we can replace the edges $v_i x$ and $v_j y$ with the non-edges $xy$ and $v_i v_j$ to obtain a new realization of $\pi$.
Iteratively performing this process for each nonadjacent pair of vertices in $S'$ yields a realization of $\pi$ in which $S'$ induces a complete graph.
Consequently $\pi$ is potentially $H$-graphic.

Finally, we must consider the case where $d_{k-\lst} < d_k + \lst(k  + 2) < 2k + \lst(k + 2)$.
Observe that $\pi$ is degree sufficient for $K_{k-\lst} \vee \overline{K}_{\lst}$ since $t \geq k-\lst$ and $f(H) > k$.
Let $\pi_2^w= (g_1, \ldots, g_{n'})$ be the want sequence of $\pi$ for $K_{k-\lst} \vee \overline{K}_{\lst}$ where $n' = n-(k-\lst)$.
Since constructing the want sequence increases each term by at most $k-\lst$, the facts that $d_{k-\lst+1}<2k + \lst(k + 2)$, $d_{k}< 2k-3$, and $d_{2k}< k-1$  imply that $\pi_2^w$ has the following properties:
\begin{itemize}
\item $g_i < 3k + (k+1) \lst $ \hspace{1pt} for $1 \leq i \leq \lst$,
\item $g_i < 3k- \lst$ \hspace{36pt} for $\lst +1 \leq i \leq \lst + k$, and
\item $g_i < 2k-\lst $ \hspace{36pt} for $\lst + k + 1 \leq i \leq n'$.
\end{itemize}

\begin{cl}\label{cl:WantGraphic}
 The sequence $\pi_2^w$ is graphic.
\end{cl}
\noindent\textit{Proof of Claim 1.} Note that $n' = n-(k-\ell^*)>6k^2+(\ell^*)^2k$.

Tripathi and Vijay~\cite{TriVij} showed that the Erd\H{o}s-Gallai criteria (Theorem \ref{thm:ErdGal}) need only be checked for certain values of $p$: it suffices to check all $p \leq s$, where $s$ is the largest integer for which $d_s \geq s-1$, or to check only those values of $p$ for which $d_p$ is strictly greater than $d_{p+1}$.  We will use the Erd\H{o}s-Gallai criteria and this observation to show that $\pi_2^w$ is graphic.

Since $g_i < 2k-\lst$ for large enough $i$, we only need to check the inequalities for indices up to $2k$. We can write the right hand side of  \eqref{eq:ErdGal} as
\[
p(p-1) + \sum_{i=p+1}^r p + \sum_{i=r+1}^{n'} g_i,
\]
where $r \geq p+1$ is the largest index such that $g_r \geq p$ but $g_{r+1} < p$.
This then simplifies to
\begin{align*}
p(r-1) + \sum_{i=r+1}^{n'} d_i & \geq p(r-1) + n'-r  \\
	& =r(p-1) - p + n' \\
	& \geq (p+1)(p-1) - p + n' \\
	&  = p^2 - p -1 + n'
\end{align*}
So we need to show that $n' +p^2-p-1 \geq \sum_{i=1}^p g_i$ for each $p \leq 2k$.

First suppose $p \leq \lst$. Then $$\sum_{i=1}^p g_i < p(3k+\lst(k+1)) \leq \lst 3k +(\lst)^2(k+1).$$ Since $n' \geq 6k^2 + (\lst)^2 k$, the desired inequality holds.

If $\lst + 1 \leq p \leq \lst +k$, then $$\sum_{i=1}^p g_i \leq 3\lst k +(\lst)^2( k+1) + (p-\lst)(3k-\lst) \leq 2 \lst k + (\lst)^2 (k+1) + 3k^2.$$ For $p$ in this range, $$n'+p^2-p-1 \geq 6k^2 + (\lst)^2 k + (\lst + 1)^2 -1 = 6k^2 +(\lst)^2(k+1) + 2\lst,$$ so the inequality holds.

Finally, if $\lst+ k + 1 \leq p \leq 2k$, then $$\sum_{i=1}^p g_i \leq 2\lst k + (\lst)^2 (k+1) + 3k^2 +(p-\lst -k)(2k-\lst) \leq 5k^2 + (\lst)^2 k + 2(\lst)^2 - \lst k.$$ Now, $n' +p^2 - p - 1 > 6k^2 + (\lst)^2 k + 4k^2$, so the Erd\H{o}s-Gallai inequality is satisfied.

Thus Claim 1 is proved. \\

We can now use a realization of $\pi_2^w$ to create a copy of $K_{k-\lst} \vee \overline{K}_{\lst}$ in a realization of $\pi$.
Since $H \subseteq K_{k-\lst} \vee \overline{K}_{\lst}$, this implies that $\pi$ is potentially $H$-graphic.
\end{pf}

\subsection{Proof of Lemma \ref{biclique}}\label{section:biclique_proof}

Kleitman and Wang gave the following generalization of the graphicality criteria due independently to Havel and Hakimi \cite{Hak, Hav}.

\begin{thm}[Kleitman and Wang, \cite{KleitWang}]\label{thm:KW}
Let $\pi = (d_1, \ldots, d_n)$ be a nonincreasing sequence of nonnegative integers, and let $i \in [n]$. If $\pi_i$ is the sequence defined by
\[
\pi_i =
\begin{cases}
(d_1 - 1, \ldots, d_{d_i}-1, d_{d_i+1}, \ldots, d_{i-1}, d_{i+1}, \ldots, d_n) & \text{ if } d_i < i \\
(d_1-1, \ldots, d_{i-1}-1, d_{i+1}-1, \ldots, d_{d_i+1}-1, d_{d_i+2}, \ldots, d_n) & \text{ if } d_i \geq i,
\end{cases}
\]
then $\pi$ is graphic if and only if $\pi_i$ is graphic.
\end{thm}

Let $\pi'$ be the sequence resulting from sorting $\pi_i$ in nonincreasing order, and call $\pi'$ the \textit{residual sequence} obtained by {\it laying off $d_i$}.  Repeated application of Theorem \ref{thm:KW} yields an efficient algorithm to test for graphicality.  ~\\

\begin{proof}[Proof of Lemma \ref{biclique}.]~\\

\noindent\textbf{Idea of the proof:}
The proof of Lemma \ref{biclique} is based on a careful analysis of repeated applications of the Kleitman-Wang algorithm (Theorem \ref{thm:KW}).  Observe that when laying off a term $d_i$ from a graphic sequence, the $d_i$ terms of highest degree, aside from $d_i$, are each reduced by 1.  If there are many terms of the same value that will be reduced, the order in which these reductions occur does not matter.  In particular, provided we reduce the correct number of terms, we may reduce any of the terms equal to $d_{d_i}$ and will get the same residual sequence.   This fact is the key to constructing a realization of $\pi$ that contains $K_{k-r,r}$, as referenced in the statement of Lemma \ref{biclique}.

Kleitman-Wang provides a means by which this realization can be constructed on the vertex set $V=\{v_1,\ldots,v_n\}$ so that the vertices $v_j$ have degree $d_j$ for $j=1,\ldots,n$.
The vertex $v_j$ is associated with the $j$th term in $\pi$.  When $d_j$ is laid off, the resulting sequence, $\pi'$, is a graphic sequence.
We can use it to construct a graph on $V-\{v_j\}$ with (the reordered) $\pi'$ as its degree sequence.
The vertex $v_j$ is then added, adjacent to the first $d_j$ members of $\{v_1,\ldots,v_{j-1},v_{j+1},\ldots,v_n\}$.
In this way, when we lay off a term $d_j$ of $\pi$, we will say that the vertices associated with the terms that are reduced are assigned to the neighborhood of $v_j$.
Repeating this process, we will create the desired realization of $K_{k-r,r}$.

 The problem with this procedure is that applying it more than once requires that each of the degree sequences must be reordered, which makes keeping track of the vertices that are assigned to a particular neighborhood difficult.

For clarity, we will often abuse terminology and say we lay off vertex $v_j$ to mean we lay off the term of $\pi$ whose value corresponds to the degree of $v_j$.
This makes sense when we think about laying off a term $d_j$ of $\pi$ as assigning a set of vertices to the neighborhood of $v_j$.
We will lay off at most $r(k+r+1)$ vertices with the aim of obtaining just $r$ of them whose neighborhood contains $\{v_1,\ldots,v_{k-r}\}$. Our parameters are chosen just for this purpose.  The entries in $\{d_{k+1},\ldots,d_n\}$ that have value in $\{k-r,\ldots,k-1\}$ will be the candidates for entries to lay off.  Because $d_{k-r}-d_{k}\geq r(k+2)$, we can guarantee that, for each of the degree sequences that result from the laying off, the entries that correspond to $v_1,\ldots,v_{k-r}$ will always stay within the first $k-1$ entries.~\\

\noindent\textbf{Terms and definitions:} Now we proceed to prove that the procedure outlined above does indeed produce the graph we want.
We will lay off entries of $\pi$ corresponding to vertices $v_{a_1},v_{a_2},\ldots,v_{a_p},\ldots$, where $v_{a_p}$ will be determined at step $p$.
Let $V_0=V$ and for $p=1,2,\ldots$, let $V_p=V_{p-1}-v_{a_p}$.
The neighborhood we assign to $v_{a_p}$, which we will call $N_p$, is a subset of $V_p$.
We will call the process of laying off $v_{a_1},\ldots,v_{a_{r(k+2)}}$ the Laying-off Algorithm.

For $p=0,1,2,\ldots$, we define $\hatd_p(v_i)$ to be the \textit{remaining degree} of $v_i$ after $v_{a_1},\ldots,v_{a_p}$ are laid off.
That is, for every vertex $v_i$, $\hatd_0(v_i)=d_i$ and for $p=1,2,\ldots,r(k+2)$, we have $\hatd_p(v_i):=d_i-\left|\{N_j\colon 1\le j\le p\text{ and }v_i\in N_j\}\right|$. Iteratively,
\begin{equation*}\label{eq:dhat}
\hatd_p(v_i) = \begin{cases}
                  \hatd_{p-1}(v_i)   & \text{ if }v_i\notin N_p \\
                  \hatd_{p-1}(v_i)-1 & \text{ if }v_i\in N_p.
               \end{cases}
\end{equation*}
To determine which vertex $v_{a_p}$ to lay off, for $p =1,2,\ldots$, we define $S_{p-1}\subset V_{p-1}$ to be the set of all vertices $w\in V_{p-1}$ for which $\hatd_{p-1}(w)\in\{k-r,\ldots,k-1\}$.
Then choose $v_{a_p}$ to be a vertex in $S_{p-1}$ for which $\hatd_{p-1}(v_{a_p})$ is minimum.
Let $\ell_{p-1}= \hat d_{p-1}(v_{a_p})$; this is the number of vertices that will be assigned to $N_p$.
Note that the neighborhood of $v_{a_p}$ may not consist solely of the vertices in $N_p$.
In particular, if $\hat d_{p-1}(v_{a_p}) < \hat d_0(v_{a_p})$, then $v_{a_p}$ is in $N_{p'}$ for some $p'<p$. Thus, the neighborhood of $v_{a_p}$ in our final graph contains $v_{a_{p'}}$ although $v_{a_{p'}}$ is not in $N_p$.

The natural ordering on $V=V_0$ is simply $(v_1,\ldots,v_n)$.
This corresponds to the nonincreasing order of $\pi$.
We say that $v_i$ \textit{naturally precedes} $v_j$ if $i<j$, and will write $v_i\propto v_j$.
We will define $\pi_p$ to be the sequence given by each $\hatd_p(v_i)$, for all $v_i\in V_p$, that is nonincreasing and, when equality holds, to obey the natural ordering.
That is, $\hatd_p(v_i)$ precedes $\hatd_p(v_j)$ in $\pi_p$ if either (a) $\hatd_p(v_i)>\hatd_p(v_j)$, or (b) $\hatd_p(v_i)=\hatd_p(v_j)$ and $i<j$.
This is simply the degree sequence obtained from $\pi$ by $p$ iterations of the Kleitman-Wang algorithm; thus, $\pi_p$ is graphic.

Observe that in defining $\pi_p$, we have prescribed the order of the terms based on the vertices with which they are associated.
This is because we need to keep track of not only the remaining degree of a vertex but also the position of that vertex in $\pi_p$.
To make this precise, let $\tau_p$ be a function from $\{1,\ldots,|V_p|\}\rightarrow V_p$ in which $\tau_p(j)$ is the vertex in the $j$th position in the order defined by $\pi_p$, and let $T_p$ be the sequence $\tau_p(1), \ldots, \tau_p(n-p)$.
Thus, $T_p$ is simply the sequence of vertices of $V_p$, ordered according to the position of their remaining degree in $\pi_p$.
A subsequence $\tau_p(b_1),\ldots,\tau_p(b_m)$ of $T_p$ is \textit{consistent} if $\tau_p(b_i)\propto\tau_p(b_j)$ for all $b_i<b_j$. In essence, this means that all vertices in the subsequence are in order by index, from lowest to highest. We say that $T_p$ itself is consistent if $\tau_p(1),\ldots,\tau_p(n-p)$ is consistent.


In the Kleitman-Wang algorithm, when the term $d_i$ is laid off it is first removed from the sequence; then the first $d_i$ terms of the resulting sequence are each reduced by one.
To incorporate this into the Laying-off Algorithm, we define $\hat{\pi}_{p-1}$ to be $\pi_{p-1}$ with the term associated with $v_{a_p}$ removed. So, $\htau_{p-1}$ and $\hat T_{p-1}$ are the corresponding order function and sequence of vertices. ~\\

\noindent\textbf{Finding the ``neighborhoods'' $N_p$:}
Now we can describe our modification of the Kleitman-Wang algorithm more precisely.
At step $p$ of the Laying-off Algorithm, we choose $N_p$ in the following way:
\begin{enumerate}
   \item If $T_{p-1}$ is consistent, then simply let $N_p$ be the first $\ell_{p-1}$ vertices in $\hat T_{p-1}$.

   \item If $T_{p-1}$ is not consistent but $\hatd_{p-1}(\htau_{p-1}(\ell_{p-1}))>\hatd_{p-1}(\htau_{p-1}(\ell_{p-1}+1))$, then we again let $N_p$ consist of the first $\ell_{p-1}$ vertices in $T_{p-1}$.

   \item If $T_{p-1}$ is not consistent but $\hatd_{p-1}(\htau_{p-1}(\ell_{p-1}))=\hatd_{p-1}(\htau_{p-1}(\ell_{p-1}+1))$, then $N_p$ consists of all vertices $w \in V_{p-1}$ for which $\hat d_{p-1}(w) > \hat d_{p-1}(\htau_{p-1}(\ell_{p-1}))$, and the vertices $x$ with the highest index for which $\hat d_{p-1}(x)=\hat d_{p-1}(\htau_{p-1}(\ell_{p-1}))$.

\end{enumerate}

In words, what we do is identify the vertices with largest $\hatd_p$ values and reduce their values by 1. If $T_{p-1}$ is not consistent and we cannot reduce all of those with the same value, we reduce those with largest index (i.e., those that come later in the ordering). When $T_{p-1}$ is consistent, we still take the first $\ell_{p-1}$ vertices, even if all of those with the same value are not reduced. We will say that $N_p$ is \textit{good} if $\{v_1,\ldots,v_{k-r}\}\subseteq N_p$. The existence of at least $r$ vertices among $\{v_{k+1},\ldots,v_n\}$ such that laying off each gives a good $N_p$ will yield the $K_{k-r,r}$ we seek.~\\

\noindent\textbf{An example:} Let us do a small example to illustrate the way the Laying-off Algorithm works.
\begin{sloppypar}
Begin with the graphic sequence $\pi = (9,9,9,9,8,8,7,7,7,7,4,4,4,4)= (d(v_1), \ldots, d(v_{14}))$.
For the purposes of this example, we will only identify the neighborhoods of the vertices with degree 4.
\end{sloppypar}
\begin{description}
\item[Step 1]
Since the original ordering of vertices is consistent, we assign the neighborhood of $v_{14}$ to be $N_1=\{v_1,v_2,v_3,v_4\}$. The new sequence is $\pi_1=(8^6,7^4,4^3)$, and since $\hatd_1(v_4)\geq\hatd_1(v_5)$, there is no reordering of vertices and $T_1$ is consistent.
\item[Step 2]
Since $T_1$ is consistent, we can simply assign the set $N_2=\{v_1,v_2,v_3,v_4\}$ to the neighborhood of $v_{13}$. Now $\pi_2=(8,8,7^8,4^2)$. However, the vertices are no longer in their original order; the sequence $T_2$ is: $v_5,v_6,v_1,v_2,v_3,v_4,v_7,v_8,v_9,v_{10},v_{11},v_{12}$.
\item[Step 3]
Since $T_2$ is not consistent, we must consider $\hat d_2(\htau_2(4))$. Since $\hatd_2(\htau_2(4))=\hatd_2(\htau_2(5))$, we cannot simply assign the four highest-degree vertices to $N_3$.
We begin with $N_3=\{v_5,v_6\}$, the two highest-degree vertices. Then we need two more vertices, so we take the two vertices of degree $\hatd_2(\htau_2(4))=7$ that have the highest index, that is $v_9$ and $v_{10}$. So $N_3=\{v_5,v_6,v_9,v_{10}\}$. This leaves $\pi_3=(7^8,6^2,4)$, and $T_3$ is consistent.
\item[Step 4]
Since $T_3$ is consistent, we let $N_4 = \{v_1, v_2, v_3, v_4\}$. Then $\pi_4=(7^4,6^6)$.
\end{description}
Observe that at each step $\pi_i$ is exactly the sequence we would get after $i$ iterations of the Kleitman-Wang algorithm if a term of value 4 is laid off each time.~\\

\noindent\textbf{Proof that the Laying-off Algorithm gives $r$ good neighborhoods:} Now we will show that this process does create $r$ vertices among $\{v_{k+1},\ldots,v_n\}$ that have good neighborhoods. We begin with several claims that develop useful properties of the Laying-off Algorithm, in particular the key observation that $\ell_p \geq \ell_{p-1}$ for all $p \leq rk$.
Then, we show that at each iteration of the algorithm, the sequence $T_p$ has a certain structure that allows us to easily count the number of iterations needed to find $r$ good $N_p$s. ~\\

\noindent\textbf{Claim 1.} If $p\leq r(k+1)$ and $i<j$, then $\hatd_p(v_j)\leq \hatd_p(v_i)+1$.~\\

\noindent\textit{Proof of Claim 1.} If $\hatd_p(v_j)\geq\hatd_p(v_i)+2$ then, since $d_0(v_i)\geq d_0(v_j)$, there exists a $p'$ such that $\hatd_{p'-1}(v_j)=\hatd_{p'-1}(v_i)$, $v_i\in N_{p'}$ and $v_j\not\in N_{p'}$, and there also exists a $p''$ such that $\hatd_{p''-1}(v_j)=\hatd_{p''-1}(v_i)+1$, $v_i\in N_{p''}$ and $v_j\not\in N_{p''}$.  But such a $p''$ cannot exist because if $\hatd_{p''-1}(v_j)>\hatd_{p''-1}(v_i)$, then $v_i\in N_{p''}$ implies $v_j$ is also in $N_{p''}$.  This contradiction proves Claim 1.~\hfill~$\Box$~\\

\noindent\textbf{Claim 2.} If $p\leq r(k+1)$ and $j\leq k-r$, then $\hatd_{p-1}(v_j)\geq k$. In addition, if $S_{p-1}\cap N_p \neq\emptyset$, then $N_p$ is good.~\\

\noindent\textit{Proof of Claim 2.} If $v_j$ is not laid off, then $d_j$ decreases by at most 1 at each step and so $\hatd_{p-1}(v_j)\geq d_j-(p-1)$.
Because $d_{k-r}\geq d_k+r(k+2)$, we have the following:
\[ \hatd_{p-1}(v_j)\geq d_j-(p-1)\geq d_{k-r}-(p-1)\geq d_k+r(k+2)-(p-1)\geq d_k+r. \]
The conditions on the sequence force $d_k\geq k-r$, giving $\hatd_{p-1}(v_j)\geq d_k+r\geq k$. As a result, if $S_{p-1}\cap N_p\neq\emptyset$, then $N_p$ must contain every vertex with remaining degree greater than $\hatd_p(v_{a_p})\leq k-1$. This includes all of $\{v_1,\ldots,v_{k-r}\}$ and so $N_p$ must be good.~\hfill~$\Box$~\\

Let $g_p$ denote the number of good neighborhoods $N_{p'}$ with $p'\leq p$. We may assume that $g_p\leq r-1$ for all $p\leq r(k+1)$. Otherwise, we would have $r$ good neighborhoods, hence our copy of $K_{k-r,r}$. In particular, by Claim 2 we can assume that there are at most $r-1$ values of $p$ for which $S_{p-1}\cap N_p\neq\emptyset$. ~\\

\noindent\textbf{Claim 3.} If $p\leq r(k+1)$, then $|S_{p-1}\cap N_p|\leq r-1$ and $|S_{p-1}|>r(k+r+1)-g_{p-1}(r-1)-(p-1)\geq 2r$. In addition, every $v\in N_p$ has $\hatd_{p-1}(v)$ at least as large as the least value of $\hatd_{p-1}$ among members of $S_{p-1}$.~\\

\noindent\textit{Proof of Claim 3.} Consider the vertex $v_{a_p}$.  It has degree at most $k-1$ when it is laid off. By Claim 2, there are at least $k-r$ vertices $v_j$ with $\hatd_{p-1}(v_j)\geq k$ and so $|S_{p-1}\cap N_p|\leq (k-1)-(k-r)=r-1$.  Because a vertex will only leave the set $S_p$ if it has been laid off or assigned to the neighborhoods of enough other vertices that its remaining degree is too low,
\[
|S_{p-1}|\geq r(k+r+1)-\left|\bigcup_{j=1}^{p-1}\{S_{j-1}\cap N_j\}\right|-(p-1)\geq r(k+r+1)-g_{p-1}(r-1)-(p-1).
\]
Since $g_{p-1}\leq r-1$, we have $|S_{p-1}|\geq r(k+r+1)-(r-1)^2-(p-1)\geq 2r$.  If we include the vertices $\{v_1,\ldots,v_{k-r}\}$, there are a total of at least $k$ vertices $w$ for which $\hatd_{p-1}(w)$ is at least the minimum value of $\hat d_{p-1}$ among the members of $S_{p-1}$.  This proves Claim 3.~\hfill~$\Box$~\\

\noindent\textbf{Claim 4.} If $\ell_p<\ell_{p-1}$ for some $p\leq rk$, then at most $r$ more iterations of the Laying-off Algorithm will create the desired copy of $K_{k-r,r}$. \\

\noindent\textit{Proof of Claim 4.} By definition, $\ell_{p-1}=\hatd_{p-1}(v_{a_p})$ and $\ell_p=\hatd_p(v_{a_{p+1}})$. Since $v_{a_p}$ was chosen to minimize $\hatd_{p-1}$ among $S_{p-1}$, we know that $\hatd_{p-1}(v_{a_{p+1}})\geq\ell_{p-1}\geq k-r$.

Since $\hatd_{p}(v_{a_{p+1}})\geq\hatd_{p-1}(v_{a_{p+1}})-1$, we get $\ell_p=\hatd_{p}(v_{a_{p+1}})\geq\ell_{p-1}-1$. Thus, $\ell_p=\ell_{p-1}-1$. This means that $\hatd_{p-1}(v_{a_{p+1}})=\ell_{p-1}$ and $v_{a_{p+1}}\in N_p$. Since $v_{a_{p+1}}\in S_{p-1}$ as well, Claim 2 gives that $N_p$ is good.

Further, as $\ell_{p-1}$ is also the minimum remaining degree of any vertex in $S_{p-1}$, Claim 3 gives that every vertex $w$ in $N_p$ has $\hatd_{p-1}(w)\geq\ell_{p-1}$. Since $v_{a_{p+1}}\in N_p$, we conclude that $\hatd_{p-1}(\htau_{p-1}(\ell_{p-1}))=\ell_{p-1}$. Since $\hatd_{p-1}(v_{k-r})>\hatd_{p-1}(v_{a_{p+1}})$, there are at most $\ell_{p-1}-1$ vertices $w$ with $\hatd_{p-1}(w)\geq\hatd_{p-1}(v_{k-r})$.

So, if we can show that $\ell_{p'}\geq\ell_{p-1}-1$ for all $p'$ such that $p\leq p'\leq p+(r-g_p)$, then each $N_{p'}$ is good and we have the desired $K_{k-r,r}$ in at most $r$ more steps. From Claim 3, there are at most $r-2$ vertices in $S_{p-1}\cap N_p$ that have remaining degree larger than $\hatd_{p-1}(v_{a_{p+1}})=\ell_{p-1}$. Also from Claim 3, $|S_{p-1}|>r(k+r+1)-g_{p-1}(r-1)-(p-1)$. Thus, there are at least
\[ r(k+r+1)-g_{p-1}(r-1)-(p-1)-(r-2)> r(r-g_{p-1}) \]
vertices of remaining degree equal to $\ell_{p-1}$ in $S_{p-1}$. Since Claim 3 gives that $|S_{p'-1}\cap N_{p'}|\leq r-1$, for all $p'\geq p$, each of the next $r-g_p$ iterations of the Laying-off Algorithm will remove at most $r$ vertices from $S_{p-1}$ which have remaining degree equal to $\ell_{p-1}$.

Hence there is always a vertex in $S_{p-1}$ with degree equal to $\ell_{p-1}$. Thus, no vertex with degree $\ell_{p-1}-1$ will be placed into $N_{p'}$, and Claim 4 is proved.~\hfill~$\Box$~\\

We can thus assume that $\ell_{p} \geq \ell_{p-1}$ for all $p \leq rk$. ~\\

Now we are prepared to examine the structure of the sequence $T_p$.
Claim 5 below is the main observation, that even when the Laying-off Algorithm results in an inconsistent sequence, the sequence that results is of a very specific form. Thus,  the Laying-off Algorithm ensures that the number of iterations between consistent sequences is less than $k$.

To show this, we say the sequence $T_p$ is of \textit{proper form} if there is a partition of $V_p$ into four ordered sets $\tau_p^{(1)}$, $\tau_p^{(2)}$, $\tau_p^{(3)}$ and $\tau_p^{(4)}$ (where the order is inherited from $\tau_p$) such that $\hat d_p$ is constant on each of $\tau_p^{(2)}$ and $\tau_p^{(3)}$ and, when $i<j$ and $v_i,v_j\in V_p$, $v_j$ precedes $v_i$ if and only if $v_j\in\tau_p^{(2)}$ and $v_i\in\tau_p^{(3)}$.
By Claim 1, we know that in this case $\hat d_p(v_j) \leq \hat d_p(v_i)+1$.
Note that this allows for $\tau_p^{(2)}$ and $\tau_p^{(3)}$ to be empty, in which case $T_p$ is consistent.

We will abuse notation to let $\cup$ represent the ``concatenation'' of ordered sets; that is, $\tau_p^{(i)} \cup \tau_p^{(j)}$ is also an ordered set, where the elements of $\tau_p^{(i)}$ precede those of $\tau_p^{(j)}$, and within each set the original order is maintained.
Thus, if $T_p$ is of proper form, both $\tau_p^{(1)}\cup\tau_p^{(2)}$ and $\tau_p^{(3)}\cup\tau_p^{(4)}$ are consistent.  For a sequence $T_p$ that is of proper form, the \textit{inconsistency} of $T_p$ is $\left|\tau_p^{(2)}\cup\tau_p^{(3)}\right|$. A consistent sequence has inconsistency zero.~\\

\noindent\textbf{Claim 5.} For all $p\in\{0,\ldots,rk-1\}$, $T_p$ is of proper form. If $T_{p-1}$ is consistent or has inconsistency at least $k$, then $N_p$ is good. If $T_p$ is inconsistent, then $|\tau_p^{(1)}\cup\tau_p^{(3)}|\leq\ell_{p-1}$. If $T_{p-1}$ has positive inconsistency, then either
\begin{itemize}
   \item $T_p$ is consistent (and $N_p$ is good),
   \item $T_p = \hat T_{p-1}$ and $N_p$ is good, or 
   \item $T_p$ has inconsistency strictly less than the inconsistency of $T_{p-1}$.
\end{itemize}~\\

\noindent\textit{Proof of Claim 5.}
We will prove the claim by induction on $p$.

If $p=0$, then $T_p=T_0$ is consistent.  Moreover $N_{p+1}$ is good because it is simply the first $\ell_p$ entries of $\hat T_p$, which must contain $v_1,\ldots,v_{k-r}$. In fact, this is true for any consistent $T_p$ and this will be our base case for the induction.

We assume the statement of the claim is true for $T_0,\ldots,T_{p-1}$.~\\

\noindent\textbf{Case 1:} $T_{p-1}$ is consistent.~\\
The set $N_p$ is good because it is simply the first $\ell_{p-1}$ entries of $\hat T_{p-1}$, which must contain $v_1,\ldots,v_{k-r}$.
If $T_p$ is consistent, then it is, by definition, of proper form.

If $T_p$ is not consistent, then $\hat d_{p-1}(\htau_{p-1}(\ell_{p-1})) = \hat d_{p-1}(\htau_{p-1}(\ell_{p-1}+1))$. We can partition $\hat T_{p-1}$ into $\htau_{p-1}^{(1)} \cup \htau_{p-1}^{(4)L} \cup \htau_{p-1}^{(4)R}$, where $\htau_{p-1}^{(4)L}$ contains all vertices with remaining degree exactly $\hat d_{p-1}(\htau_{p-1}(\ell_{p-1}))$ and $\htau_{p-1}^{(4)R}$ contains those with lower remaining degree.  We can further partition $\htau_{p-1}^{(4)L}$ into $\htau_{p-1}^{(4)L_1}$ and $\htau_{p-1}^{(4)L_2}$, where $\htau_{p-1}^{(4)L_1}$ contains all vertices of $\htau_{p-1}^{(4)L}$ that are included in $N_p$, and $\htau_{p-1}^{(4)L_2}$ consists of those that are not.

Now $T_p$ can be partitioned into
\begin{align*}
   \tau_p^{(1)} = &  \htau_{p-1}^{(1)}\\
   \tau_p^{(2)} = & \htau_{p-1}^{(4)L_2} \\
   \tau_p^{(3)} = & \htau_{p-1}^{(4)L_1}, \text{ and } \\
   \tau_p^{(4)} = & \htau_{p-1}^{(4)R},
\end{align*}
and it is of proper form.
Clearly $|\tau^{(1)}_p\cup\tau^{(3)}_p|=\ell_{p-1}$.~\\

Observe that if $T_{p-1}$ is consistent and $T_p$ is not, then $\{v_1,\ldots,v_{k-r}\}$ is contained in $\tau_p^{(1)}\cup\tau_p^{(2)}\cup\tau_p^{(3)}$.~\\

\noindent\textbf{Case 2:} $T_{p-1}$ is not consistent.~\\
\indent Recall that $\ell_{p-1}=|N_p|$, the number of vertices in $V_p$ that are reduced by one when a vertex of $T_{p-1}$ is laid off.
The effect of the Laying-off Algorithm on $T_p$ depends on the value of $\ell_{p-1}$.

Note that $\ell_{p-1} \leq |\htau_{p-1}^{(1)}|$ is not possible because Claim 4 allows us to assume that $\ell_{p-1}\geq\ell_{p-2}$. Since $\ell_{p-2}\geq |\tau_{p-1}^{(1)}\cup\tau_{p-1}^{(3)}|>|\htau_{p-1}^{(1)}|$, this is a contradiction.

With this information, we can show that if the inconsistency of $T_{p-1}$ is at least $k$, then $N_p$ is good. The largest $k$ entries of $T_{p-1}$ are in $\tau_{p-1}^{(1)}\cup\tau_{p-1}^{(2)}\cup\tau_{p-1}^{(3)}$ and $N_p$ contains all of $\htau_{p-1}^{(1)}$. Because there are $r(k+r+1)$ vertices eligible to be laid off from $\{v_{k+1},\ldots,v_n\}$, and we've laid off at most $rk$, the vertices $v_{k-r},\ldots,v_k$ will not be laid off.
If $v_k$ is in $\htau_{p-1}^{(3)}$, then its value is at most $d_k-1$, and if $v_k$ is in $\htau_{p-1}^{(2)}$, then its value is at most $d_k$. But the degree of each of $v_1,\ldots,v_{k-r}$ is at least $d_{k-r}-kr\geq d_k+r(k+2)-kr>d_k$.  So, each of $v_1,\ldots,v_{k-r}$ are in $\htau_{p-1}^{(1)}$ and will be in $N_p$ as long as the inconsistency is at least $k$. ~\\

\noindent\textbf{Case 2a:} $|\htau_{p-1}^{(1)}|<\ell_{p-1}<|\htau_{p-1}^{(1)}\cup\htau_{p-1}^{(2)}|$.~\\
\indent In this case, we can partition $\htau_{p-1}^{(2)}$ into two pieces: $\htau_{p-1}^{(2)L}$ and $\htau_{p-1}^{(2)R}$.  The members of $\htau_{p-1}^{(2)R}$ are reduced when $v_{a_p}$ is laid off, but those of $\htau_{p-1}^{(2)L}$ are not.

After reordering, we obtain the following
\begin{align*}
   \tau_p^{(1)} = & \htau_{p-1}^{(1)}, \\
   \tau_p^{(2)} = & \htau_{p-1}^{(2)L}, \\
   \tau_p^{(3)} = & \htau_{p-1}^{(3)}, \\
   \tau_p^{(4)} = & \htau_{p-1}^{(2)R}\cup\htau_{p-1}^{(4)}.
\end{align*}
Moreover, $\ell_{p-1}\geq\ell_{p-2}\geq |\htau_{p-1}^{(1)}\cup\htau_{p-1}^{(3)}|=|\htau_p^{(1)}\cup\htau_p^{(3)}|$. In addition, the inconsistency of $T_p$ is $|\tau_p^{(2)}\cup\tau_p^{(3)}|=|\htau_{p-1}^{(2)L}\cup\htau_{p-1}^{(3)}|$, which is strictly less than the inconsistency of $T_{p-1}$ because $\htau_{p-1}^{(2)L}$ is a strict subset of $\htau_{p-1}^{(2)}$.~\\

To proceed through the next cases, we must partition $\htau_{p-1}^{(4)}$ into two pieces: $\htau_{p-1}^{(4)L}$ and $\htau_{p-1}^{(4)R}$. The members of $\htau_{p-1}^{(4)L}$ have the same remaining degree as those in $\htau_{p-1}^{(3)}$, and those of $\htau_{p-1}^{(4)R}$ have smaller remaining degree (either or both of these may be empty).~\\

\noindent\textbf{Case 2b:} $|\htau_{p-1}^{(1)}\cup\htau_{p-1}^{(2)}|\leq\ell_{p-1}\leq |\htau_{p-1}^{(1)}\cup\htau_{p-1}^{(2)}|+|\htau_{p-1}^{(4)L}|$.~\\
\indent In this case, the values of $\htau_{p-1}^{(1)}\cup\htau_{p-1}^{(2)}$ as well as some of $\htau_{p-1}^{(4)L}$ are reduced. Since the members of $\htau_{p-1}^{(2)}\cup\htau_{p-1}^{(3)}$ (and the unreduced values of $\htau_{p-1}^{(4)L}$) now have the same value, reordering results in $T_p$ being a consistent sequence.~\\

\noindent\textbf{Case 2c:} $|\htau_{p-1}^{(1)}\cup\htau_{p-1}^{(2)}|+|\htau_{p-1}^{(4)L}| <\ell_{p-1} <|\htau_{p-1}^{(1)}\cup\htau_{p-1}^{(2)}\cup\htau_{p-1}^{(3)}\cup\htau_{p-1}^{(4)L}|$.~\\

In this case, we can partition $\htau_{p-1}^{(3)}$ into two pieces: $\htau_{p-1}^{(3)L}$ and $\htau_{p-1}^{(3)R}$.  The members of $\htau_{p-1}^{(3)R}$ are reduced but those of $\htau_{p-1}^{(3)L}$ are not.

After reordering, we obtain the following
\begin{align*}
   \tau_p^{(1)} = & \htau_{p-1}^{(1)}\cup\htau_{p-1}^{(3)L}, \\
   \tau_p^{(2)} = & \htau_{p-1}^{(2)}, \\
   \tau_p^{(3)} = & \htau_{p-1}^{(3)R}, \\
   \tau_p^{(4)} = & \htau_{p-1}^{(4)}.
\end{align*}
Moreover, $\ell_{p-1}\geq\ell_{p-2}\geq |\htau_{p-1}^{(1)}\cup\htau_{p-1}^{(3)}|=|\htau_p^{(1)}\cup\htau_p^{(3)}|$. In addition, the inconsistency of $T_p$ is $|\tau_p^{(2)}\cup\tau_p^{(3)}|=|\htau_{p-1}^{(2)}\cup\htau_{p-1}^{(3)R}|$, which is strictly less than the inconsistency of $T_{p-1}$ because $\htau_{p-1}^{(3)R}$ is a strict subset of $\htau_{p-1}^{(3)}$.~\\

\noindent\textbf{Case 2d:} $|\htau_{p-1}^{(1)}\cup\htau_{p-1}^{(2)}\cup\htau_{p-1}^{(3)}\cup\htau_{p-1}^{(4)L}|\leq\ell_{p-1}$.~\\
\indent In this case, no rearranging is necessary: the order of the vertices in  $T_p$ is the same as the order in $\hat T_{p-1}$.

Because the only vertices out of order are in $\htau_{p-1}^{(2)}\cup\htau_{p-1}^{(3)}$, $N_p$ will contain all of the first $\ell_{p-1}$ vertices. Since $\ell_{p-1}\geq k-r$, the neighborhood $N_p$ must contain $\{v_1,\ldots,v_{k-r}\}$ and thus be good.

This concludes the proof of Claim 5.~\\

Given Claim 5, the proof of Lemma~\ref{biclique} follows easily.  There can be at most $k-1$ neighborhoods that are not good between consecutive good neighborhoods. So after $(r-1)k+1$ iterations of the procedure, there will be $r$ good neighborhoods.

\end{proof}


\begin{thebibliography}{99}

\bibitem{BK} N. Bushaw and N. Kettle. Tur\'an numbers of Multiple Paths and Equibipartite Trees, {\it Combin. Probab. Comput.} {\bf 20} (2011), 837--853

\bibitem{ChenLiYin} G. Chen, J. Li and J. Yin, A variation of a classical
Tur\'{a}n-type extremal problem, {\it European J. Comb.}
{\bf 25} (2004), 989--1002.

\bibitem{CGPW} G. Chen, R. Gould, F. Pfender and B. Wei, Extremal Graphs for Intersecting Cliques, {\it J. Combin. Theory Ser. B} {\bf 89} (2003), 159--181.

\bibitem{Erd} P. Erd\H{o}s, On the graph theorem of Tur\'{a}n (Hungarian), \textit{Mat. Lapok.} \textbf{21} (1970), 249--251.

\bibitem{EFGG} P. Erd\H{o}s, Z. F\"{u}redi, R. Gould and D. Gunderson, Extremal Graphs for Intersecting Triangles, {\it J. Combin. Theory Ser. B} {\bf 64} (1995), 89--100.

\bibitem{ErdGal} P. Erd\H{o}s and T. Gallai, Graphs with prescribed degrees of vertices (Hungarian), \textit{Mat. Lapok.} \textbf{11}, 264--274, 1960.

\bibitem{EJL} P. Erd\H{o}s, M.S. Jacobson, and J. Lehel, Graphs Realizing the Same Degree Sequence and their Respective Clique Numbers. \textit{Graph Theory, Combinatorics and Applications} (eds. Alavi, Chartrand, Oellerman and Schwenk), Vol. 1, 1991, 439--449.

\bibitem{ESS} P. Erd\H{o}s and M. Simonovits, A Limit Theorem in Graph Theory, {\it Studia Sci. Math. Hungar.} {\bf 1} (1966), 51--57.

\bibitem{ES} P. Erd\H{o}s and A. Stone, On the Structure of Linear Graphs, {\it Bull. Amer. Math. Soc.} {\bf 52} (1946), 1087--1091.

\bibitem{F} M. Ferrara, Graphic Sequences with a Realization Containing a Union of Cliques, {\it Graphs Comb.} {\bf 23} (2007), 263--269.

\bibitem{FLMW} M. Ferrara, T. LeSaulnier, C. Moffatt, P. Wenger. On the Sum Necessary to Ensure that a Degree Sequence is Potentially $H$-graphic, submitted.

\bibitem{FS} M. Ferrara and J. Schmitt,  A General Lower Bound for Potentially $H$-Graphic Sequences, {\it SIAM J. Discrete Math.} {\bf 23} (2009), 517--526.


\bibitem{Hak} S.L. Hakimi, On the realizability of a set of integers as degrees of vertices of a graph, \textit{J. SIAM Appl. Math}, \textbf{10} (1962), 496--506.

\bibitem{Hav} V. Havel, A remark on the existence of finite graphs (Czech), \textit{\v Casopis P\v est. Mat.} \textbf{80} (1955), 477--480.

\bibitem{KezLeh} A.E. K\'ezdy and J. Lehel, Degree sequences of graphs with prescribed clique size, in: Y. Alavi et al. (Ed.), \textit{Combinatorics, Graph Theory, and Algorithms}, Vol. 1, New Issues Press, Kalamazoo, Michigan, 1999, 535--544.

\bibitem{KleitWang} D. Kleitman and D. Wang, Algorithms for constructing graphs and digraphs with given valences and factors, \textit{Discrete Math.} \textbf{6} (1973) 79--88.


\bibitem{LiXia}  J. Li and Z. Song, The smallest degree sum
that yields potentially $P_k$-graphical sequences, {\it J. Graph Theory}
{\bf 29} (1998), 63--72.

\bibitem{LiSongLuo}  J. Li, Z. Song, and R. Luo, The
Erd\H{o}s-Jacobson-Lehel conjecture on potentially $P_k$-graphic
sequences is true, {\it Science in China, Ser. A},
{\bf 41} (1998), 510--520.

\bibitem{LiYin} J. Li and J. Yin, An extremal problem on potentially $K_{r,s}$-graphic sequences, \textit{Discrete Math.} \textbf{260} (2003), 295--305.

\bibitem{LiYin2} J. Li and J. Yin, An extremal problem on
potentially $K_{r,s}$-graphic sequences, {\it Discrete Math.} {\bf 260}
(2003), 295--305.

\bibitem{PikTar} O. Pikhurko and A. Taraz, Degree sequences of $F$-free graphs, \textit{Electronic J. Combin.} \textbf{12} (2005), R69, 12pp.

\bibitem{Rao} A.R. Rao, An Erd\H{o}s-Gallai type result on the clique number of a realization of a degree sequence. (unpublished)

\bibitem{TriVij} A. Tripathi and S. Vijay, A note on a theorem of Erd\H{o}s and Gallai, \textit{Discr. Math.} \textbf{265}, 417--420, 2003.

\bibitem{Tur41} P. Tur\'{a}n, Eine Extremalaufgabe aus der Graphentheorie, {\it Mat. Fiz. Lapook} {\bf 48} (1941), 436--452.

\bibitem{Yin} J.H. Yin, A Rao-type characterization for a sequence to have a realization containing a split graph, \textit{Discrete Math.} \textbf{311} (2011), 2485--2489.

\bibitem{YinRao} J.H. Yin, A short constructive proof of A. R. Rao's characterization of potentially $K_{r+1}$-graphic sequences, {\it Discrete Applied Math.} {\bf 160} (2012), 352--354.

\bibitem{ZZ} I.\ E.\ Zverovich and V.\ E.\ Zverovich, Contributions to the theory of graphic sequences, {\it Discr. Math.} {\bf 103} (1992), 293--303.

\end{thebibliography}
\end{document}